\documentclass[a4paper,reqno]{amsart}
\def\@logofont{\footnotesize}
\def\@setaddresses{\par
  \nobreak \begingroup
  \footnotesize
  \def\author##1{\nobreak\addvspace\bigskipamount}%
  \def\\{\par\nobreak}%
  \interlinepenalty\@M
  \def\address##1##2{\begingroup
    \par\addvspace\bigskipamount\indent
    \@ifnotempty{##1}{(\ignorespaces##1\unskip) }%
    {\scshape\ignorespaces##2}\par\endgroup}%
  \def\curraddr##1##2{\begingroup
    \@ifnotempty{##2}{\nobreak\indent\curraddrname
      \@ifnotempty{##1}{, \ignorespaces##1\unskip}\/:\space
      ##2\par}\endgroup}%
  \def\email##1##2{\begingroup
    \@ifnotempty{##2}{\nobreak\indent\emailaddrname
      \@ifnotempty{##1}{, \ignorespaces##1\unskip}\/:\space
      \ttfamily##2\par}\endgroup}%
  \def\urladdr##1##2{\begingroup
    \def~{\char`\~}%
    \@ifnotempty{##2}{\nobreak\indent\urladdrname
      \@ifnotempty{##1}{, \ignorespaces##1\unskip}\/:\space
      \ttfamily##2\par}\endgroup}%
  \addresses
  \endgroup
}
\renewcommand*\subjclass[2][2010]{%
  \def\@subjclass{#2}%
  \@ifundefined{subjclassname@#1}{%
    \ClassWarning{\@classname}{Unknown edition (#1) of Mathematics
      Subject Classification; using '2000'.}%
  }{%
    \@xp\let\@xp\subjclassname\csname subjclassname@#1\endcsname
  }%
}

\usepackage{graphicx,amsmath,amssymb}
\usepackage{algorithm}
\usepackage[noend]{algpseudocode}

\makeatletter
\def\BState{\State\hskip-\ALG@thistlm}
\makeatother
\allowdisplaybreaks

\theoremstyle{plain}
\newtheorem{theorem}{Theorem}[section]

\newtheorem{lemma}[theorem]{Lemma}
\newtheorem{corollary}[theorem]{Corollary}
\newtheorem{conjecture}[theorem]{Conjecture}

\theoremstyle{definition}
\newtheorem{definition}[theorem]{Definition}
\newtheorem{example}[theorem]{Example}

\newtheorem{remark}[theorem]{Remark}
\def\supp{\mathrm{supp}}

\begin{document}

\title[A note on  matchings in abelian groups]
{A note on   matchings in abelian groups}

\author[M. Aliabadi]{Mohsen Aliabadi}
\address{Mohsen Aliabadi \\ 
Department of Mathematics, Statistics, and Computer Science\\ University of  Illinois\\  851 S. Morgan St, Chicago, IL 60607, USA}
\email{maliab2@uic.edu}

\author[S. Soleimany Dizicheh]{Shiva Soleimany Dizicheh}
\address{Shiva Soleimany Dizicheh\\ 
Department of Computing Science\\ University of Algebra\\  Edmonton, Canada}
\email{soleiman@ualberta.ca}

\thanks{Keywords and phrases. acyclic matching, acyclicity sequence, weak acyclic matching property.}
\thanks{2010 Mathematics Subject Classification. Primary: 05D15; Secondary: 11B75, 20D60.}

\begin{abstract}
The question of finding sets of monomials which are removable from a generic homogeneous polynomial through a linear change in its variables was  raised by E. K. Wakeford in 1916. This linear algebra question motivated J. Losonczy to define the concept of acyclic matchings in $\mathbb{Z}^n$, and later in abelian groups.  In this paper, we give a constructive approach to study the acyclic matchings in cyclic groups. We also introduce the notion of weakly matched subsets and investigate its relation with matchings in abelian groups.
\end{abstract}

\maketitle

\section{Introduction}
Let $B$ be a finite subset of the abelian group $G$ which does not contain the neutral element. For any subset $A$ in $G$ with the same cardinality as $B$, a {\it matching} from $A$ to $B$ is defined to be a bijection $f:A\to B$ such that for any $a\in A$ we have $a+f(a)\not\in A$. For any matching $f$ as above, the associated multiplicity function $m_f:G\to \mathbb{Z}_{\geq0}$ is defined via the rule:
\begin{align}
\forall x\in G,\quad m_ f(x)=\#\{a\in A:\, a+ f(a)=x\}.
\end{align}
A matching $ f:A\to B$ is called {\it acyclic} if for any matching $g:A\to B$, $m_f=m_g$ implies $f=g$. The notion of matchings in abelian groups was introduced by Fan and Losonczy in \cite{4} in order to generalize a geometric property of lattices in Euclidean space.   The motivation to study acyclic matchings is their relations with an old problem of Wakeford concerning canonical forms for symmetric tensors \cite{6}. This notion has been investigated in literature in different aspects (e.g., the existence of acyclic matchings in subsets of abelian torsion-free groups and acyclic groups of prime order.)  In this paper, we will investigate acyclic matchings in certain subsets of  abelian groups. We will   introduce the concept of strongly acylically matched subsets of abelian groups and will provide a family of such subsets. Finally, we introduce the notion of weakly matchable subsets and prove that it implies the existence of matchings in usual sense.

\section{Preliminary}
Throughout this paper,  we assume that $G$ is an abelian group and $A$, $B$ are two non-empty finite subsets of $G$ with the same cardinality and $0\not\in B$. Following 
Losonczy  in \cite{5}, we say that $G$ possesses the {\it acyclic matching property} if for every pair $A$ and $B$ of finite subsets of $G$ with $\# A=\# B$ and $0\not\in B$ there is at least one matching from $A$ to $B$. We say that $G$ has the {\it weak acyclic matching property} if for every pair $A$ and $B$ of $G$ with $\# A=\# B$ and $A\cap (A+B)=\emptyset$, there is at least one acyclic matching from $A$ to $B$. It was proven in \cite{3} that $\mathbb{Z}^n$ has the acyclic matching property. Later, this result was generalized by Losonczy  to abelian torsion-free groups \cite{5}. Also, it was shown in \cite{1} that there are infinitely many primes $p$ for which $\mathbb{Z}/p\mathbb{Z}$ does not have the acyclic matching property. We conjecture that all abelian groups possess the weak acyclic matching property:
\begin{conjecture}\label{c2.1}
Let $G$ be an abelian group. Then $G$ possesses the weak acyclic matching property.
\end{conjecture}
In the next section, we will provide a constructive approach to investigate the  weak acyclic matching property in $G$. However, Conjecture \ref{c2.1} will still remain unsolved.

\section{A constructive approach to Conjecture \ref{c2.1}}
For any matching $f:A\to B$, define {\it support} of $f$ as $\supp(f)=\{x\in G:\, m_f(x)>0\}$. We associate the finite ordered sequence  $x_1^{(f)},x_2^{(f)},\ldots,x_{n_f}^{(f)}$ to $f$ such that:
\begin{itemize}
\item[(i)]
$x_i=m_f(x)$, for some $x\in \supp(f)$,
\item[(ii)]
the sequence is non-increasing, i.e. $x_1^{(f)}\geq x_2^{(f)}\geq\cdots\geq x_{n_f}^{(f)}$,
\item[(iii)]
$n_f=\#\supp(f)$.
\end{itemize}
We call the above sequence the {\it acyclicity sequence} of $f$.

Denote the set of all matchings from $A$ to  $B$ by $\mathcal{M}(A,B)$. Define 
\begin{align*}
C_1^{(A,B)}&=\max\left\{ x_1^{(f)}:\, f\in\mathcal{M}(A,B)\right\},\\
F^{(1)}_{(A,B)}&=\left\{f\in\mathcal{M}(A,B):\, x_1^{(f)}=C_1^{(A,B)}\right\}.
\end{align*}
If $\# F^{(1)}_{(A,B)}=1$ and $F_{(A,B)}^{(1)}=\{f\}$, then $f$ is an acyclic matching as for any $g\in\mathcal{M}(A,B)$ with $f\neq g$ we have $m_f\neq m_g$. If $\# F_{(A,B)}^{(1)}>1$, define 
\begin{align*}
C_2^{(A,B)}&=\max\left\{ x_2^{(f)}:\, f\in F_{(A,B)}^{(1)}\right\},\\
F^{(2)}_{(A,B)}&=\left\{f\in F_{(A,B)}^{(1)}:\, x_2^{(f)}=C_2^{(A,B)}\right\}.
\end{align*}
Similar to the previous case, if $\# F_{(A,B)}^{(2)}=1$ and $F_{(A,B)}^{(2)}=\{f\}$, then $f$ is acyclic. Otherwise,  $\# F^{(2)}_{(A,B)}>1$ and we define
\begin{align*}
C_3^{(A,B)}&=\max\left\{ x_3^{(f)}:\, f\in F_{(A,B)}^{(2)}\right\},\\
F^{(3)}_{(A,B)}&=\left\{f\in F_{(A,B)}^{(2)}:\, x_3^{(f)}=C_3^{(A,B)}\right\}.
\end{align*}
Continuing in this manner, we obtain sequences $\left\{C_i^{(A,B)}\right\}$ and $\left\{F_{(A,B)}^{(i)}\right\}$ with the following recurrence relations:
\begin{align*}
C_i^{(A,B)}&=\max\left\{ x_i^{(f)}:\, f\in F_{(A,B)}^{(i-1)}\right\},\\
F^{(i)}_{(A,B)}&=\left\{f\in F_{(A,B)}^{(i-1)}:\, x_i^{(f)}=C_i^{(A,B)}\right\},
\end{align*}
where $i>1$. 

Since $\mathcal{M}(A,B)<\infty$, the process described above will terminate after finitely many steps. Assume that $F_{(A,B)}^{(t)}\neq\emptyset$ and $F_{(A,B)}^{(t+1)}=\emptyset$. We call $t$ the {\it acyclicity index} of $\mathcal{M}(A,B)$.
\begin{conjecture}\label{c3.1}
Let $F^{(t)}_{(A,B)}$ be as above, that is, $t$ is the acyclicity index of $\mathcal{M}(A,B)$. If $A\cap(A+B)=\emptyset$, then every $f\in F^{(t)}_{(A,B)}$ is acyclic.
\end{conjecture}
\begin{remark}\label{r3.2}
Conjecture \ref{c3.1} implies Conjecture \ref{c2.1}.
\end{remark}
\begin{example}\label{e3.3}
Consider the following subsets of $\mathbb{Z}/14\mathbb{Z}$:
\[A=\{1,3,5,7\}\quad\text{and}\quad B=\{1,3,7,9\}\]
One  can check that $A\cap (A+B)=\emptyset$. Therefore every bijection  from $A$ to $B$ is a matching. The followings are all 24 matchings from $A$ to $B$ and their acyclicity sequences:

{\begin{table}[h!]
\centering
\caption{Acyclicity table of $\mathcal{M}(A,B)$}
\resizebox{\textwidth}{!}{\begin{tabular}{|c|c|c|c|c|c|c|c|}
\hline
Matching & Rule & Support & Acyclicity sequence & Matching & Rule & Support & Acyclicity sequence\\
\hline
$f_1$ & $\begin{array}{c}1\to 1\\[-.8mm] 3\to 3\\[-.8mm] 5\to 7\\[-.8mm] 7\to 9\end{array}$ & $\{2,6,12\}$ & $2,1,1$ &
$f_{13}$ & $\begin{array}{c}1\to 7\\[-.8mm] 3\to 1\\[-.8mm] 5\to 3\\[-.8mm] 7\to 9\end{array}$ & $\{2,4,8\}$ & $2,1,1$
\\
\hline
$f_2$ & $\begin{array}{c}1\to 1\\[-.8mm] 3\to 3\\[-.8mm] 5\to 9\\[-.8mm] 7\to 7\end{array}$ & $\{0,2,6\}$ & $2,1,1$&
$f_{14}$ & $\begin{array}{c}1\to 7\\[-.8mm] 3\to 1\\[-.8mm] 5\to 9\\[-.8mm] 7\to 3\end{array}$ & $\{0,4,8,10\}$ & $1,1,1,1$
\\
\hline
$f_3$ & $\begin{array}{c}1\to 1\\[-.8mm] 3\to 9\\[-.8mm] 5\to 3\\[-.8mm] 7\to 7\end{array}$ & $\{0,2,8,12\}$ & $1,1,1,1$&
$f_{15}$ & $\begin{array}{c}1\to 7\\[-.8mm] 3\to 3\\[-.8mm] 5\to 1\\[-.8mm] 7\to 9\end{array}$ & $\{2,6,8\}$ & $2,1,1$
\\
\hline
$f_4$ & $\begin{array}{c}1\to 1\\[-.8mm] 3\to 9\\[-.8mm] 5\to 7\\[-.8mm] 7\to 3\end{array}$ & $\{2,10,12\}$ & $2,1,1$&
$f_{16}$ & $\begin{array}{c}1\to 7\\[-.8mm] 3\to 3\\[-.8mm] 5\to 9\\[-.8mm] 7\to 1\end{array}$ & $\{0,6,8\}$ & $2,1,1$
\\
\hline
$f_5$ & $\begin{array}{c}1\to 1\\[-.8mm] 3\to 7\\[-.8mm] 5\to 3\\[-.8mm] 7\to 9\end{array}$ & $\{2,8,10\}$ & $2,1,1$&
$f_{17}$ & $\begin{array}{c}1\to 7\\[-.8mm] 3\to 9\\[-.8mm] 5\to 1\\[-.8mm] 7\to 3\end{array}$ & $\{6,8,10,12\}$ & $1,1,1,1$
\\
\hline
$f_6$ & $\begin{array}{c}1\to 1\\[-.8mm] 3\to 7\\[-.8mm] 5\to 9\\[-.8mm] 7\to 3\end{array}$ & $\{0,2,10\}$ & $2,1,1$&
$f_{18}$ & $\begin{array}{c}1\to 7\\[-.8mm] 3\to 9\\[-.8mm] 5\to 3\\[-.8mm] 7\to 1\end{array}$ & $\{8,12\}$ & $3,1$
\\
\hline
$f_7$ & $\begin{array}{c}1\to 3\\[-.8mm] 3\to 1\\[-.8mm] 5\to 7\\[-.8mm] 7\to 9\end{array}$ & $\{2,4,12\}$ & $2,1,1$&
$f_{19}$ & $\begin{array}{c}1\to 9\\[-.8mm] 3\to 1\\[-.8mm] 5\to 3\\[-.8mm] 7\to 7\end{array}$ & $\{0,4,8,10\}$ & $1,1,1,1$
\\
\hline
$f_8$ & $\begin{array}{c}1\to 3\\[-.8mm] 3\to 1\\[-.8mm] 5\to 9\\[-.8mm] 7\to 7\end{array}$ & $\{0,4\}$ & $2,2$&
$f_{20}$ & $\begin{array}{c}1\to 9\\[-.8mm] 3\to 1\\[-.8mm] 5\to 7\\[-.8mm] 7\to 3\end{array}$ & $\{4,10,12\}$ & $2,1,1$
\\
\hline
$f_9$ & $\begin{array}{c}1\to 3\\[-.8mm] 3\to 7\\[-.8mm] 5\to 9\\[-.8mm] 7\to 1\end{array}$ & $\{0,4,8,10\}$ & $1,1,1,1$&
$f_{21}$ & $\begin{array}{c}1\to 9\\[-.8mm] 3\to 3\\[-.8mm] 5\to 1\\[-.8mm] 7\to 7\end{array}$ & $\{0,6,10\}$ & $2,1,1$
\\
\hline
$f_{10}$ & $\begin{array}{c}1\to 3\\[-.8mm] 3\to 7\\[-.8mm] 5\to 1\\[-.8mm] 7\to 9\end{array}$ & $\{2,4,6,10\}$ & $1,1,1,1$&
$f_{22}$ & $\begin{array}{c}1\to 9\\[-.8mm] 3\to 3\\[-.8mm] 5\to 7\\[-.8mm] 7\to 1\end{array}$ & $\{6,8,10,12\}$ & $1,1,1,1$
\\
\hline
$f_{11}$ & $\begin{array}{c}1\to 3\\[-.8mm] 3\to 9\\[-.8mm] 5\to 1\\[-.8mm] 7\to 7\end{array}$ & $\{0,4,6,12\}$ & $1,1,1,1$&
$f_{23}$ & $\begin{array}{c}1\to 9\\[-.8mm] 3\to 7\\[-.8mm] 5\to 3\\[-.8mm] 7\to 1\end{array}$ & $\{8,10\}$ & $2,2$
\\
\hline
$f_{12}$ & $\begin{array}{c}1\to 3\\[-.8mm] 3\to 9\\[-.8mm] 5\to 7\\[-.8mm] 7\to 1\end{array}$ & $\{4,8,12\}$ & $2,1,1$&
$f_{24}$ & $\begin{array}{c}1\to 9\\[-.8mm] 3\to 7\\[-.8mm] 5\to 1\\[-.8mm] 7\to 3\end{array}$ & $\{6,10\}$ & $3,1$
\\
\hline
\end{tabular}}
\end{table}}

Using the acyclicity table of $\mathcal{M}(A,B)$, we  get $C_1^{(A,B)}=3$, $C_2^{(A,B)}=1$,  $F^{(1)}_{(A,B)}=\{f_{18},f_{24}\}$,   $F^{(2)}_{(A,B)}=\{f_{18},f_{24}\}$ and $F^{(3)}_{(A,B)}=\emptyset$. As we expected, both $f_{18}$ and $f_{24}$ are acyclic matchings. 
\end{example}

\begin{example}\label{ex3.4}
Running two Python codes (see Algorithm 1 and Algorithm 3 in Section 4)  for the existence of acyclic matchings by   computing $C_i^{(A,B)}$ and $F^{(i)}_{(A,B)}$, we get the following results for subsets $A=\{0,1,2,3,12,13,14, 15\}$ and  $B=\{4,5,6,7,8,16,17,18\}$ of $\mathbb{Z}/23\mathbb{Z}$ (Note that there are 40320 matchings from $A$ to $B$.)\\
$C_1^{(A,B)}=7$ and $F_1^{(A,B)}=\{f_1,f_2,f_3\}$, where $f_i$'s are given as follows: 
{\scriptsize\begin{table}[h!]
\centering
\begin{tabular}{|c|c|c|}
\hline
$f_1$ & $f_2$ & $f_3$\\
\hline
 $\begin{array}{l}0\to 7\\[-.8mm] 1\to 6\\[-.8mm] 2\to 5\\[-.8mm] 3\to 4\\[-.8mm] 12\to18\\[-.8mm] 13\to 17\\[-.8mm] 14\to16\\[-.8mm] 15\to 8\end{array}$ &  $\begin{array}{l}0\to 8\\[-.8mm] 1\to 7\\[-.8mm] 2\to6\\[-.8mm] 3\to 5\\[-.8mm] 12\to4\\[-.8mm] 13\to 18\\[-.8mm] 14\to17\\[-.8mm] 15\to 16\end{array}$ &  $\begin{array}{l}0\to 8\\[-.8mm] 1\to 18\\[-.8mm] 2\to 17\\[-.8mm] 3\to 16\\[-.8mm] 12\to7\\[-.8mm] 13\to 6\\[-.8mm] 14\to5\\[-.8mm] 15\to 4\end{array}$\\
\hline
\end{tabular}
\end{table}}

In the following table, supports and acyclicity sequences of the above matchings are shown.
{\scriptsize\begin{table}[h!]
\centering
\begin{tabular}{|c|c|c|c|}
\hline
&$f_1$ & $f_2$ & $f_3$\\
\hline
Support & $\{0,7\}$ & $\{8,16\}$ & $\{8,19\}$\\
\hline
Acyclicity Sequence & $7,1$ & $7,1$ & $7,1$\\
\hline
\end{tabular}
\end{table}}

We also have $C_2^{(A,B)}=1$ and $F_2^{(A,B)}=\{f_1,f_2,f_3\}$. Since $C_1^{(A,B)}+C_2^{(A,B)}=8=\#A$, then $F_3^{(A,B)}=\emptyset$ and so the acyclicity index of $\mathcal{M}(A,B)=2$. Our simulation results indicate that $f_1$, $f_2$ and $f_3$ are acyclic matchings as we expected.
\end{example}
\begin{theorem}\label{t3.5}
If $\# A=\# B=n>1$, $A\cap (A+B)=\emptyset$ and $B\cup\{0\}$ is a subgroup of $G$, then
\begin{itemize}
\item[(i)]
The acyclicity sequence of any $f\in\mathcal{M}(A,B)$  is $\underbrace{1,1,\ldots,1}_{n-\mathrm{times}}$.
\item[(ii)]
Every $f\in\mathcal{M}(A,B)$ is acyclic.
\end{itemize}
\end{theorem}
\begin{proof}
Assume to the contrary, $m_f(x)>1$, for some $x\in G$. Then, there exist distinct  $a,a'\in A$ for which $x=a+f(a)=a'+f(a')$. We have $a=a'+f(a')-f(a)\in A+B$ as $f(a')-f(a)\in B$. This contradicts $A\cap(A+B)=\emptyset$. This gives (i).

Now, choose two matchings $f,g\in\mathcal{M}(A,B)$. Choose $a\in A$ such that $f(a)\neq g(a)$. From  the previous part, we have  $m_f(a+f(a))=1$. We claim that $m_g(a+f(a))=0$. If not, then  there exists $a'\in A$ such that $a+f(a)=a'+g(a')$. This implies $a=a'+g(a')-f(a)\in A+B$ as $g(a')-f(a)\in B$. This contradicts $A\cap (A+B)=\emptyset$. Therefore, $m_g(a+f(a))=0$ and so $m_f\neq m_g$. This  implies that $f$ is acyclic. The proof is complete.
\end{proof}
\begin{remark}
If $f$ is a matching whose acyclicity sequence only includes 1, then $f$ is not necessarily acyclic. Our simulation results  in Example \ref{ex3.4} show that there exist  2436 matchings whose acyclicity sequences only include 1. However, only 8 of them are  acyclic. The acyclicity sequences of $f,g\in\mathcal{M}(A,B)$ given as follows only include 1. 

{\scriptsize\begin{table}[h!]
\centering
\begin{tabular}{|c|c|}
\hline
$f$ & $g$\\
\hline
 $\begin{array}{l}0\to 18\\[-.8mm] 1\to 16\\[-.8mm] 2\to 17\\[-.8mm] 3\to 5\\[-.8mm] 12\to4\\[-.8mm] 13\to 7\\[-.8mm] 14\to8\\[-.8mm] 15\to 6\end{array}$ &  $\begin{array}{l}0\to 18\\[-.8mm] 1\to 16\\[-.8mm] 2\to17\\[-.8mm] 3\to 5\\[-.8mm] 12\to4\\[-.8mm] 13\to 8\\[-.8mm] 14\to6\\[-.8mm] 15\to 7\end{array}$\\
\hline
\end{tabular}
\end{table}}

Note that we have $m_f=m_g$ and this implies that $f$ and $g$ are not acyclic.
\end{remark}
Our simulation results show that for $A$ and $B$ with  $\#A=\#B$ and $A\cap(A+B)=\emptyset$ there exists at least one acyclic matching from $A$ to $B$. So we have the following conjecture:
\begin{conjecture}\label{c3.7}
Let $G$ be an abelian group. If $A$ and $B$ are two finite subsets of $G$ with $A\cap (A+B)=\emptyset$, then there exists $f\in\mathcal{M}(A,B)$ such that $f$ is acyclic and its acyclicity sequence only includes 1.
\end{conjecture}

Note that Conjecture \ref{c3.7} provides a constructive approach to investigate Conjecture \ref{c2.1} and indeed it  implies Conjecture \ref{c2.1}.

\begin{definition}
We say that $A$ is {\it acylically matched} to $B$ if there exists an acyclic matching  from $A$ to $B$. We also say that $A$ is {\it strongly acylically matched} to $B$ if $A$ is acylically matched to $B$ and every matching from $A$ to $B$ is acyclic. 
\end{definition}
\begin{remark}
If $A$ and $B$ satisfy the conditions of Theorem \ref{t3.5},  then $A$ is strongly acylically matched to $B$.
\end{remark}
\begin{example}
Let $A$ and $B$ be as Example \ref{e3.3}. Then $A$ is not strongly acylically matched to $B$ as  $f_9$ is not acyclic $(m_{f_9}=m_{f_{19}})$.
\end{example}
\begin{remark}
That $A$ is acylically matched to $B$ does not imply that $A$ is strongly acylically matched to $B$. For  example our simulation results in Example \ref{ex3.4} show that even though $A$ is acylically matched to $B$, $A$ is not strongly acylically matched to $B$ as the matching $f$ given below is not acyclic 

{\scriptsize\begin{table}[h!]
\centering
\begin{tabular}{|c|}
\hline
$f$\\
\hline
$\begin{array}{l}0\to 18\\[-.8mm] 1\to 16\\[-.8mm] 2\to 17\\[-.8mm]3\to5\\[-.8mm]12\to4\\[-.8mm]13\to7\\[-.8mm]14\to8\\[-.8mm]15\to6
\end{array}$\\
\hline
\end{tabular}
\end{table}}
\end{remark}
\begin{remark}
As we mentioned, Theorem \ref{t3.5} provides a family of sets that are strongly acylically matched. Further investigations of other such families could turn out to be worthwhile.
\end{remark}
\begin{remark}
If $A$ is strongly acylically matched to $B$ and $f\in\mathcal{M}(A,B)$, the acyclicity sequence of $f$ does not only include 1 necessarily. For example, consider the subsets $A=\{2,4\}$ and $B=\{3,1\}$ of $\mathbb{Z}$. Then $A$ is strongly acylically matched  to $B$ but the acyclic matching $f\in\mathcal{M}(A,B)$ given via the rule $2\to 3$ and $4\to 1$ has 2 in its acyclicity sequence.
\end{remark}
\begin{conjecture}
If $A$ is acylically matched to $B$ and the acyclicity sequence of every $f\in\mathcal{M}(A,B)$ only includes 1,  then $A$ is strongly acylically matched to $B$.
\end{conjecture}

\section{Weak Matchings}
This section  aims to investigate substructures of matchings. We start with the following definition.
\begin{definition}
Let $G$ be an abelian group and $A$ and $B$ be two non-empty finite subsets of $G$ with the same cardinality $n$ and $B$ does not contain the neutral element. We say $A$ is weakly matched to $B$ of order $m$ ($1\leq m<n$) if for every subset  $A'$ of $A$ with $\# A'=m$, there exists a subset $B'$ of $B$ such that $A'$ is matched to $B'$, in usual sense.
\end{definition}
\begin{remark}\label{r}
If $A$ is matched to $B$, then one can verify easily that $A$ is weakly matched to $B$ of order $m$, for any $1\leq m<n$.
Consider an arbitrary non-empty subset $A'$ of $A$ with $\# A'=m$ and set $B':=f(A')$, where $f$ is a matching from $A$ to $B$. Then the mapping $g=f\big|_{A'}$ is a matching from $A'$ to $B'$ and so $A$ is weakly matched to $B'$ of order $m$.
\end{remark}
As we verified above, having two matchable subsets at hand, it is straightforward to show that the given sets are weakly matchable of any order.  The converse is not so obvious though. In what follows we prove a special case; weakly matchable subsets of order $n-1$ are matchable. We start with the following lemma.
\begin{lemma}\label{l}
Let $A$ and $B$ be non-empty finite subsets of an arbitrary group $G$, and assume that $\# A\leq\# B$ and that $A+B=A$. Then $B$ is a subgroup of $G$ and $A$ is a left coset of $B$.
\end{lemma}
\begin{proof}
If $b$ is in $B$ then the mapping $\begin{array}{rl}\varphi: A&\to A+b\\ a&\mapsto a+b\end{array}$ is injective and thus $\#(A+b)=\# A$. Since $A+b\subseteq A+B=A$ and $A$ is finite, it follows that $A+b=A$. Now let  $X=\{x\in G\mid A+x=A\}$. Then $B\subseteq X$ and $X$ is a subgroup of $G$. Also $A+X=A$. If $a\in A$, then $a+X$ is a left coset of the subgroup $X$, and thus $\#(a+X)=\# X$ and we have $a+X\subseteq A+X=A$. Then $\# B\geq \# A\geq \#(a+X)=\# X\geq \# B$, so $\# X=\# B$ and $\# A=\#(a+X)$. Since $B$ finite and is contained in $X$, and $\#B=\# X$, it follows that $B=X$, so $B$ is a subgroup as wanted. 

Since $A$ is finite and contains $a+X$, and $\# A=\# (a+X)$, it follows that $A=a+X$. We know that $X=B$, so $A=a+X=a+B$, and thus $A$ is the left coset $a+B$ of $B$. The proof is complete.
\end{proof}
\begin{corollary}\label{c}
Let $A$, $B$ and $G$ be as Lemma \ref{l}. Then $0\notin B$.
\end{corollary}
\begin{proof}
It is immediate from $B$ being  a subgroup of $G$. 
\end{proof}
Now we are ready to prove that the weakly matching property of order $n-1$ implies the matching property. We formulate this statement as follows. 
\begin{theorem}\label{t}
Let $G$ be an abelian group and $A$ and $B$ be nonempty subsets of $G$ with $\#A=\#B=n$ and $0\notin B$. Then if $A$ is weakly matched to $B$ of order $n-1$, then $A$ is matched to $B$.
\end{theorem}
\begin{proof}
We  break down the proof to two cases:

Case 1: If $A+B\neq A$, choose $y\in A+B\setminus A$. Then $y=a+b$, for some $a\in A$ and $b\in B$. Consider the subsets $A'=A\setminus \{a\}$ and $B'=B\setminus\{b\}$ of $A$ and $B$, respectively. Since $A$  is weakly matched to $B$ of order $n-1$, then $A'$ is matched to $B'$. Let $f':A'\to B'$ is a matching and define $f:A\to B$ via 
\begin{align*}
f(x)=\begin{cases} f'(x),& \mathrm{if}\; x\in A'\\ b,&\mathrm{if}\; x=a\end{cases}
\end{align*}
Then $f$ is a matching from $A$ to $B$. 

Case 2: If $A+B=A$, from   Corollary \ref{c} we conclude $0\in B$. This contradicts our assumption that $B$ does not contain the neutral element.
\end{proof}
Combining Remark \ref{r} and Theorem \ref{t} we can conclude that ``having a matching from $A$ to $B$" is equivalent to have that ``$A$ is weakly matched to $B$ of order $n-1$". It is worth pointing out that there is another concept in matching theory called {\it local matching} \cite{2}. 
Given two finite subsets $A,B\subset G$ satisfying $\#A=\#B$ and $0\notin B$, $A$ is locally matched to $B$ if for every subgroup $H\neq G$ such that 
\begin{itemize}
\item[(a)]
$A$ contains a coset of $H$ and
\item[(b)]
$H\cap B\neq\varnothing$,
\end{itemize}
there exists $A'\subset A$ such that $\#A'=\#(H\cap B)$ and a bijection $f:A'\to H\cap B$ such that for every $a\in A'$, we have $a+f(a)\notin A$.
It turns out that ``having  matching from $A$ to $B$" is equivalent to ``having local matching from $A$ to $B$" \cite[Theorem 3.1]{2}. 
So for two finite nonempty subsets $A$ and $B$ of a given abelian group $G$  with $\#A=\#B=n$ and $0\notin B$, the following statements are equivalent:
\begin{itemize}
\item[a)]
$A$ is matched to $B$,
\item[b)]
$A$ is locally matched to $B$,
\item[c)]
$A$ is weakly matched to $B$ of order $n-1$.
\end{itemize}

\section{Possible Research Problems}
First problem is whether weakly matchable subsets of lower orders are matchable. Another problem is formulating and proving linear analogue of weakly matchable subsets. More rigorously, given a field extension $K\subset L$, one may introduce analogue notion of weakly matchable $K$-subspaces $A$, $B$ of $L$ of certain orders and obtain the linear version of Theorem \ref{t}. See also \cite{2n} for some possibly helpful results in this direction.

\section{Computer Program}
In this section, we employ three algorithms to investigate acyclic matchings. In the  first  algorithm, we input $n$ and find all pairs $A$ and $B$ of subsets of $\mathbb{Z}/n\mathbb{Z}$ for which $A\cap(A+B)=\emptyset$. In the second algorithm, we find all matchings whose acyclicity sequences only include 1. It also checks whether matchings obtained in this manner are acyclic or not. The third algorithm, we inputs subsets $A$ and $B$ obtained by implementing Algorithm 1 and provides $C_i^{(A,B)}$ and $F^{(i)}_{(A,B)}$. It also checks whether elements of $F_{(A,B)}^{(t)}$ are acyclic or not. (Here, $t$ signifies the acyclicity index of $\mathcal{M}(A,B)$)
{\footnotesize\begin{algorithm}
\caption{Generating A and B}\label{euclid}
\begin{algorithmic}[1]
\Function {GenerateAandB}{$n$}
    \For{i in range($0$,$n-1$) }
        \State $ZnZ$.add($i$)
    \EndFor
    \For{$i$ in range(2,$n$/2)}
     \State $Subset(i)$=  {\Call {GiveSubsets}{$ZnZ$ $i$}}
	    \State
	    $PairofSubsets$.add(Combinations($Subset(i)$,2))
    \EndFor
    \For{$A$,$B$ in $PairofSubsets$ }
         \If  {\Call {HasWeakAcyclicity}{$A$,$B$,$n$}}
            \State $AB$.add($A$,$B$)	
        \EndIf
    \EndFor
\State return $AB$
\EndFunction
\Function{GiveSubsets}{$ZnZ$,$i$}
    \State $subset$=Combinations($ZnZ$,$i$)
    \State return $subset$
\EndFunction

\Function {HasWeakAcyclicity}{$A$,$B$,$n$}
    \State $i$=length($A$)
    \For{$k$ in range(0,$i$)}
        \For{$j$ in range(0,$i$) }
            \State $temp$=$A(k)$+$B(j)$
            \If{ $temp$ mod $n$ is in $A$ }
                \State 	return False
            \EndIf
        \EndFor
    \EndFor
    \State return True
    
\EndFunction

\Function{main}{}
\State get $n$
\State $AB$={\Call {GenerateAandB}{$n$}}
\State $ A,B {\leftarrow} a \ row\ of\ AB$
\State $support$={\Call {GiveSupport}{$A$,$B$,$n$}}
\State $sequence$= {\Call {GiveSequence}{$A$,$B$,$n$}}
\EndFunction
\end{algorithmic}
\end{algorithm}}

{\footnotesize\begin{algorithm}
\caption{Calculate matchings whose acyclicity  sequence only includes 1 }\label{euclid}
\begin{algorithmic}[1]
\Function {onesequences}{$A$,$B$,$n$}
	\State $temp$={\Call {GiveSequence}{$A$,$B$,$n$}}
	\If {sum($temp$)=length($temp$)}
		\State $oneseq$.add($temp$)
	\EndIf
\EndFunction

\Function{GiveSequence}{$A$,$B$,$n$}
    \State $supportwithrepeat$=GiveSupportwithrepeat($A$,$B$,$n$)
    \For {$i$ in range 0,$n$-1}
        \For {$j$ in range 0,length($supportwithrepeat$)}
            \If {supportwithrepeat($j$)=$i$}
               \State  $sequences(i)$+=1
            \EndIf
        \EndFor
    \EndFor
    \State remove 0s from $sequences$
	\State sort $sequences$ descending
	\State return $sequences$
\EndFunction
\Function{GiveSupport}{$A$,$B$,$n$}
    \State $support$=set(\Call{GiveSupportwithrepeat}{$A$,$B$,$n$})
    \State return $support$
\EndFunction
\Function{GiveSupportwithrepeat}{$A$,$B$,$n$}
\State $supportwithrepeat$=\Call {Givematchings}{$A$,$B$}
	\For {$match$ in $supportwithrepeat$}
		\For{$item$ in $match$}
		    \State $item$=sum($item$) mod $n$
		\EndFor

	\EndFor
	\State return $supportwithrepeat$
\EndFunction
\Function{Givematchings}{$A$,$B$}
    \State	matching=permutations(A,B)
    \State return matching
\EndFunction
\end{algorithmic}
\end{algorithm}}

\*\newpage

{\footnotesize\begin{algorithm}
\caption{Calculate $C_{(i)}^{(A,B)}$ and $F_{(A,B)}^{(i)}$ }\label{euclid}
\begin{algorithmic}[1]

\Function {CalculateF}{$sequences$,$supportwithrepeat$,$i$,$end$} 
    \If {$i$ less than $end$}
        \State $C(i)$=max($sequence(i)$ for $sequence$ in $sequences$)
        \For {$sequence$ in $sequences$}
            \If {$sequence(i)$==$C(i)$}
                \State $F(i)$.add($sequence(i)$)
                \State $sup(i)$.add($supportwithrepeat(i)$)
            \EndIf
        \EndFor
        \State $end$=min(lengh $F(i)$ members)
        \State CalculateF($F(i)$,$sup$,$i$+1,$end$)
        
    \ElsIf{$i$=$end$}
        \For {$f$ in $supportwithrepeat(i)$}
			\State isacyclic($f$)
		\EndFor

	\EndIf 
	
\EndFunction

\Function  {isacyclic}{$f$}
	\State $acyclic$=True
	\State $temp$=$supportwithrepeat$
	\State remove $f$ from $temp$
	\For { $item$ in $temp$}
		\If {AssociatedMultiplicity($item$)=AssociatedMultiplicity($f$)}
			\State $acyclic$=False
		\EndIf

	\EndFor
	\State return $acyclic$
\EndFunction

\Function{AssociatedMultiplicity}{$matching$}
    \For {$i$ in range 0,$n$-1}
        \For {$j$ in range 0,length($matching$)}
            \If {$matching(j)$=$i$}
                \State $m(i)$+=1
            \EndIf
        \EndFor
    \EndFor
    \State return $m$
\EndFunction
\end{algorithmic}
\end{algorithm}}

\end{document}